\documentclass[12pt]{article}
\usepackage{amsmath,hyperref,amsthm,amssymb,url}
\usepackage{algpseudocode,algorithm}
\usepackage{geometry}
\newtheorem{Thm}{Theorem}

\newtheorem{Lem}[Thm]{Lemma}
\newtheorem{Prop}[Thm]{Proposition}
\theoremstyle{definition}
\newtheorem{Def}[Thm]{Definition}
\newtheorem{Rmk}[Thm]{Remark}

\newcommand{\MOM}{Mom}
\begin{document}
\title{Determining hyperbolicity of
compact orientable 3-manifolds with torus boundary}
\author{Robert C. Haraway, III\thanks{Research partially supported
by NSF grant DMS-1006553.}}
\date{\today}
\maketitle

\begin{abstract}
Thurston's hyperbolization theorem
for Haken manifolds and normal surface
theory yield an algorithm to determine
whether or not a compact orientable 3-manifold
with nonempty boundary consisting of tori
admits a complete finite-volume hyperbolic metric
on its interior.

A conjecture of Gabai, Meyerhoff, and Milley
reduces to a computation using this algorithm.
\end{abstract}

\section{Introduction}
The work of J\o rgensen, Thurston, and Gromov
in the late `70s showed (\cite{Thurston82}) that the set of
volumes of orientable hyperbolic 3-manifolds
has order type $\omega^\omega$.  Cao and Meyerhoff in \cite{CM01}
showed that the first limit point is the volume
of the figure eight knot complement.
Agol in \cite{Agol10} showed that the first limit point
of limit points is the volume of the Whitehead link complement.
Most significantly for the present paper,
Gabai, Meyerhoff, and Milley in \cite{GMM09}
identified the smallest, closed, orientable
hyperbolic 3-manifold (the Weeks-Matveev-Fomenko manifold).

The proof of the last result required distinguishing hyperbolic 3-manifolds
from non-hyperbolic 3-manifolds in a large list of 3-manifolds;
this was carried out in \cite{Milley}.
The method of proof was to see whether
the \texttt{canonize} procedure of SnapPy (\cite{SnapPy})
succeeded or not; identify the successes
as census manifolds; and then examine
the fundamental groups of the 66 remaining
manifolds by hand. This method made the 
analysis of non-hyperbolic \MOM-4 manifolds, 
of which there are 762 combinatorial
types, prohibitively time-consuming.

The algorithm presented here determines
whether or not a compact 3-manifold admits
a complete finite-volume hyperbolic metric,
i.e.\,is \emph{hyperbolic}, assuming the
manifold in question has nonempty boundary
consisting of tori.

The \MOM-4s have such boundaries. The current
implementation of this algorithm using
Regina (see \cite{Regina}) classifies them, yielding
the following result.
\begin{table}
\begin{center}
\begin{tabular}{| r | r | r | r | r | r | r |}\hline
m125 & m129 & m202 & m203 & m292 & m295 & m328\\
m329 & m357 & m359 & m366 & m367 & m388 & m391\\
m412 & s441 & s443 & s503 & s506 & s549 & s568\\
s569 & s576 & s577 & s578 & s579 & s596 & s602\\
s621 & s622 & s638 & s647 & s661 & s774 & s776\\
s780 & s782 & s785 & s831 & s843 & s859 & s864\\
s880 & s883 & s887 & s895 & s898 & s906 & s910\\
s913 & s914 & s930 & s937 & s940 & s941 & s948\\
s959 & t10281 & t10700 & t11166 & t11710 & t12039 & t12044\\
t12046 & t12047 & t12048 & t12049 & t12052 & t12053 & t12054\\
t12055 & t12057 & t12060 & t12064 & t12065 & t12066 & t12067\\
t12143 & t12244 & t12412 & t12477 & t12479 & t12485 & t12487\\
t12492 & t12493 & t12496 & t12795 & t12840 & t12841 & t12842\\
v2124 & v2208 & v2531 & v2533 & v2644 & v2648 & v2652\\
v2731 & v2732 & v2788 & v2892 & v2942 & v2943 & v2945\\
v3039 & v3108 & v3127 & v3140 & v3211 & v3222 & v3223\\
v3224 & v3225 & v3227 & v3292 & v3294 & v3376 & v3379\\
v3380 & v3383 & v3384 & v3385 & v3393 & v3396 & v3426\\
v3429 & v3450 & v3456 & v3468 & v3497 & v3501 & v3506\\
v3507 & v3518 & v3527 & v3544 & v3546 &       &      \\\hline
\end{tabular}
\caption{Names of hyperbolic \MOM-4s.}\label{tbl:all}
\end{center}
\end{table}

\begin{Thm}
Table \ref{tbl:all}
constitutes the complete list of hyperbolic \MOM-4s.
\qed
\end{Thm}

\begin{proof}
Put the Python modules \texttt{fault} and \texttt{mom}
in one directory also containing the data 
\texttt{test\textunderscore mom4s\textunderscore out.txt}
from \cite{GMM11}. Then from a
\texttt{bash} prompt in a POSIX environment, run

\begin{verbatim}
python mom.py test_mom4s_out.txt | grep \| | awk -F \| '{print $1}'
\end{verbatim}

One finds that the resulting output has
138 lines, each of which is a distinct
name of a cusped manifold on a census,
either the original SnapPea census or
Thistlethwaite's more recent census
of manifolds with eight tetrahedra.
Therefore, Conjecture 5.3 from \cite{GMM11}
is correct.
\end{proof}

Of course, now we should discuss
what is in these modules.

\begin{Rmk}
The author would like to thank Tao Li, for helpful
discussions about normal surface theory and
Seifert fiberings; Dave Futer for pointing out 
an error in a previous version of the paper;
and Neil Hoffman for suggesting the use
of Berge and Gabai's work for an improved
$T^2 \times I$-homeomorphism test.
\qed
\end{Rmk}

\section{Background}

\textbf{Conventions.}
All manifolds herein are assumed to be compact and piecewise-linear.
All maps between these are assumed to be piecewise-linear
and proper (that is, such that the preimage of compacta are again compacta).
In particular, all homeomorphisms are piecewise-linear with piecewise-linear inverses.
\qed

Thurston's hyperbolicity theorem for Haken manifolds merits a succinct formulation.
Shoving some complications from the original theorem into definitions
and restricting attention to manifolds with nonempty torus boundary
yields the following theorem.

\begin{Thm}[\cite{Thurston82}, Thm. 2.3]\label{thm:common}
Let $M$ be a compact orientable 3-manifold with nonempty boundary consisting of tori.
$M$ is hyperbolic with finite volume if and only if $M$ has no faults.
\qed
\end{Thm}

The above uses the following definitions.

\begin{Def}
  A manifold is \emph{hyperbolic} when
  its interior admits a complete hyperbolic metric---a
  complete Riemannian metric of constant negative curvature.
\qed
\end{Def}

\begin{Def}
Let $s$ be an embedding of a manifold into a connected manifold $M$. 
By abuse of notation, also let $s$ denote the image of $s$ in $M$.
Suppose $s$ has codimension 1.
Pick a metric on $M$ compatible with its p.l. structure,
and let $M'$ be the path-metric completion of $M \smallsetminus s$.

When $M'$ is disconnected, $s$ \emph{separates} $M$.

When $M'$ has two connected components $N,N'$,
$s$ \emph{cuts off $N$ from $M$},
or, if $M$ is understood from context, \emph{$s$ cuts off $N$}.

If $N$ is homeomorphic to some common 3-manifold $X$, $s$ \emph{cuts off an $X$};
if, in addition, $N'$ is not homeomorphic to $X$, $s$ cuts off \emph{one} $X$.
\qed
\end{Def}

\begin{Def}
  A properly embedded surface $s$ in an orientable 3-manifold $M$ is a \emph{fault} 
  when $\chi(s) \geq 0$ and it satisfies one of the following:

  \begin{itemize}
  \item $s$ is nonorientable.
  \item $s$ is a sphere that does not cut off a 3-ball.
  \item $s$ is a disc that does not off one 3-ball.
  \item $s$ is a torus that does not cut off a
    $T^2 \times I$, and does not cut off
    a $\partial$-compressible manifold.
  \item $s$ is an annulus that does not cut off 
    a 3-ball, and does not cut off one solid torus,
    and $M$ has none of the above types of fault.
  \end{itemize}
\qed
\end{Def}

\begin{proof}
  This is a corollary of common knowledge
  surrounding Thurston's hyperbolization theorem for Haken manifolds.
  Specifically, it's commonly known that an irreducible, $\partial$-incompressible,
  geometrically atoroidal 3-manifold with nonempty boundary consisting of tori
  is either hyperbolic or Seifert-fibered.
  All Seifert-fibered spaces with at least two boundary components
  admit essential tori, which are faults.
  A Seifert-fibered space with one boundary component may admit no essential tori.
  In this case, the base orbifold $\Sigma$ is a disc with at most two cone points.
  Let $\alpha$ be a properly embedded arc which separates $\Sigma$ into discs
  with at most one cone point each; then the vertical fiber over $\alpha$
  is an annulus fault. Hence all Seifert-fibered spaces with nonempty boundary admit faults.

  Consequently, a compact orientable 3-manifold with nonempty boundary
  consisting of tori which admits no faults is irreducible, $\partial$-incompressible,
  Haken, and geometrically atoroidal, and it admits no annulus faults.
  So it must be hyperbolic.

  In fact, Thurston proved something more, namely that
  unless this manifold is $T^2 \times I$, then its metric has finite volume.
  Now, $T^2 \times I$ admits faults---non-separating annuli, in fact.
  Since we assumed the manifold had no faults, its metric must have finite volume.

  Conversely, orientable hyperbolic 3-manifolds of finite volume
  admit no orientable faults---they have no essential spheres, no compressing discs,
  no incompressible tori which aren't $\partial$-parallel,
  and no annuli which are both incompressible and $\partial$-incompressible.
  Finally, orientable hyperbolic 3-manifolds of finite volume
  don't admit any faults at all, since they admit
  no properly embedded nonorientable surfaces of nonnegative Euler characteristic.
\end{proof}

\section{A hyperbolicity algorithm}

We can turn Theorem \ref{thm:common} into an algorithm as follows.

\begin{Def}
  A 3-triangulation is a face-pairing of distinct tetrahedra.
  It is \emph{valid} when it does not identify an edge to itself backwards.
  It is \emph{material} when the link of every vertex is a sphere or disc.
  \qed
\qed
\end{Def}

\begin{Thm}\label{thm:ideal}
  Let $T$ be a valid material triangulation of a compact orientable 3-manifold $M$.
  Then $M$ has a fault precisely when $T$ has a fundamental fault.
\qed
\end{Thm}

\begin{proof}
  This is a simple consequence of standard results in normal surface theory.
  If $M$ has an embedded projective plane or essential sphere, then it
  has such a surface among fundamental normal surfaces
  (\cite{Matveev}, Theorem 4.1.12). If $M$ is irreducible and has a
  compressing disc, then it has a compressing disc among fundamental
  normal surfaces (\cite{Matveev}, proof of Theorem 4.1.13).
  If $M$ is $\partial$-irreducible and has an embedded Klein bottle
  or essential torus, then it has such a surface among fundamental
  normal surfaces (\cite{Matveev}, Lemma 6.4.7). If $M$ is $\partial$-irreducible
  and has an embedded essential annulus or M\"{o}bius band, then it
  has such a surface among fundamental normal surfaces (\cite{Matveev}, Lemma 6.4.8).
  Thus, if $M$ has an essential sphere, disc, torus, or annulus, or
  a nonorientable surface of nonnegative Euler characteristic, then
  $T$ has such a surface among its fundamental surfaces.

  Essential surfaces are the same as faults, except for essential annuli.
  An essential annulus might not be a fault, since it might cut off one
  solid torus. If $M$ is not irreducible, $\partial$-irreducible, and
  geometrically atoroidal, then $M$ has a non-annulus fault, and one
  can find such a fault among fundamental surfaces. Otherwise, $M$ is
  either hyperbolic or Seifert-fibered. If $M$ is hyperbolic
  then it has no faults. If $M$ is not hyperbolic,
  then since it is geometrically atoroidal, it is Seifert-fibered over
  $S^2$ with three exceptional fibers, over $D^2$ with
  at most two exceptional fibers, or over the annulus with at most one exceptional fiber.

  In the first case, $M$ has no faults. Consider the second case. With
  one exceptional fiber, $M$ is a solid torus and has a compressing disc.
  With two exceptional fibers, $M$ has up to isotopy only one essential
  surface, a vertical annulus separating the exceptional fibers. This
  annulus cuts off \emph{two} solid tori, not just one, so it is a fault.

  Consider now the third case. With no exceptional fibers, $M$ is $T^2\times I$.
  Up to equivalence, $M$ has only one essential annulus, $\gamma \times I$
  with $\gamma$ essential in $T^2$. This annulus cuts $M$ into one solid
  torus, but does not cut this torus \emph{off}, so this annulus is a fault.
  With one exceptional fiber, $M$ has three essential annuli up to isotopy,
  all vertical: two recurrent annuli that cut off one solid torus, and one
  nonseparating annulus $A$.

  We contend that there is a surface isotopic
  to $A$ among the fundamental surfaces, or there is
  some other, more easily detectable fault. Indeed, we can isotope $A$
  to be normal. Suppose $A = A' + A''$. If, say, $A'$ had positive
  Euler characteristic, then $A''$ would be compressible, and would
  compress to a surface isotopic to $A$, since $A$ is essential.
  Let $D$ be a complete set of compressing discs for $A''$, compressing
  it to $A$. After normalizing this surface following the shrinking moves of
  \cite{JR03}, one has a normal surface isotopic to $A$, but with smaller
  total weight. Suppose instead that $\chi(A') = \chi(A'') = 0$.
  If, say, $A'$ were an annulus, then it would necessarily also
  be a nonseparating annulus, but of less total weight.
  If, instead, neither were an annulus,
  then both would be M\"{o}bius strips. But then $M$ would have an
  embedded M\"{o}bius strip, a fault.
\end{proof}

Assuming $T$ is an ideal triangulation of a compact orientable 3-manifold $M$
with nonempty boundary consisting of tori, Algorithm \ref{alg:hyp}
determines whether or not $M$ is hyperbolic.

\begin{algorithm}
  \caption{Hyperbolicity test for link exteriors}\label{alg:hyp}
  \begin{algorithmic}[1]
    \Procedure{Hyp}{$T$}\Comment{$T$ is assumed to be valid, material, orientable, and not closed.}
    \State Let $l$ be the list of fundamental normal surfaces in $T$.
    \If{$l$ has an embedded nonorientable surface $\Sigma$ with $\chi(\Sigma) \geq 0$}
      \State \textbf{return} false
      \ElsIf{$l$ has an essential sphere}
      \State \textbf{return} false
      \ElsIf{$l$ has a compressing disc}
      \State \textbf{return} false
      \ElsIf{$l$ has an essential torus}
      \State \textbf{return} false
      \ElsIf{$l$ has an essential annulus}
      \State \textbf{return} false
      \Else \State {\textbf{return} true}
      \EndIf
    \EndProcedure
  \end{algorithmic}
\end{algorithm}

Of course, this algorithm depends upon enumerating fundamental normal
surfaces, and upon determining whether or not a normal surface is a fault.
Now, any connected compact nonorientable surface of nonnegative Euler
characteristic is a fault. A sphere is a fault when, as above, it
does not cut off one 3-ball. Regina has methods for cutting along surfaces
and determining whether or not a 3-manifold is a 3-ball. So we can
readily determine whether or not a sphere is a fault in Regina.
A disc is a fault when the same thing happens. So we could detect
whether or not a disc is a fault in Regina.
An annulus is a fault when it does not cut off a
3-ball and does not cut off one solid torus.
Regina also has a test for homeomorphism to the solid torus.
So we can test whether or not an annulus is a fault in Regina.
Finally, a torus is a fault when it does not cut
off a component admitting a compressing disc,
and does not cut off a $T^2\times I$.
Regina also has a test for admitting a compressing disc,
but does not have a test for homeomorphism to $T^2\times I$.
So to implement the above hyperbolicity algorithm,
it remains for us to implement a $T^2\times I$-homeomorphism test.
Such tests already exist in the literature, but have not been
implemented using triangulations due to their reliance on boundary patterns.

\section{A new test for homeomorphism to $T^2\times I$}

We can notice first that admitting a non-separating annulus 
is a necessary condition for being $T^2 \times I$.
We note that a further necessary condition for being $T^2 \times I$
is that splitting along any such annulus is a solid torus.
Now, if a 3-manifold $M$ split along a non-separating annulus
is a solid torus, then $M$ is a Seifert fibering with base orbifold
an annulus or M\"{o}bius band with at most a single cone point,
i.e. $M = M(0,2;r)$ or $M = M(-0,1;r)$ for some $r \in \mathbb{Q}$.
In the latter case, $M$ has only one boundary component,
so it cannot possibly be $T^2 \times I$. Thus we
may restrict our attention to the case $M = M(0,2;r)$.
Recall the following results about Seifert fiberings:

\begin{Prop}[\cite{Hatcher}, 2.1]\label{prp:h21}
Every orientable Seifert fibering is 
isomorphic to one of the models
$M(\pm g,b; s_1, \ldots, s_k).$ Any two Seifert fiberings
with the same $\pm g$ and $b$ are isomorphic
when their multisets of slopes are equal modulo 1
after removing integers, assuming $b > 0$.
\qed
\end{Prop}

\begin{Thm}[\cite{Hatcher}, 2.3]\label{thm:h23}
Orientable manifolds
admitting Seifert fiberings
have unique such fiberings 
up to isomorphism, except for
$M(0,1;s)$ for all $s \in \mathbb{Q}$ (the solid torus), 
$M(0,1;1/2, 1/2) = M(-1,1;)$
(not the solid torus), and three others without boundary.
\qed
\end{Thm}

It is quite easy to compute slopes differing mod 1
after simplifying the cusps' induced triangulations.

\begin{Prop}\label{prp:neqv}
In a triangulation of the torus $T^2$ by one vertex,
for any nontrivial element $g$ of $H_1(T^2)$,
the edges of the triangulation represent homology classes
not all equivalent mod $g$.
\qed
\end{Prop}

\begin{proof}
Suppose $v,w,x \in H_1(T^2)$ and $v+w = x$.
Let $\equiv$ denote equivalence in $H_1(T^2) \mod g$.
If $v \equiv x$, then $v + w \equiv x + w$, i.e.\, $x \equiv x + w$
(since $v + w = x$). Thus $0 \equiv w$. Now, if it were the case that
also $v \equiv w$, then also $v$ and $x$ would be $0$ mod $g$.
But then $v,w,x$ would all be multiples of $g$. However,
they generate $H_1(T^2)$, which is not cyclic. That is a
contradiction. So not all of $v,w,x$ are equivalent mod $g$.
\end{proof}

Although one could already use just the above theorems
to develop a simpler test than
the one in \cite{JT95} or \cite{Matveev} using boundary patterns,
it still involves a search for annuli, and cuts along such annuli.
Searches for annulus faults are expensive.
Neil Hoffman has kindly called my attention to the work of Berge and Gabai
on knots in solid tori (see \cite{Berge} and \cite{Gabai89}, \cite{Gabai90}).
This work enables the following simple algorithm
to determine homeomorphism to $T^2 \times I$.

\begin{Thm}
Algorithm \ref{alg:t2i} determines whether or not a compact, orientable,
3-manifold $M$ with nonempty boundary consisting of tori is $T^2 \times I$.

\begin{algorithm}
  \caption{Homeomorphism to $T^2 \times I$}\label{alg:t2i}
  \begin{algorithmic}[1]
    \Procedure{$T^2\times I$?}{$T$}
    \If{$T$ is not a homology $T^2\times I$}
    \State \textbf{return false}
    \EndIf
    \State Simplify the boundary components of $T$ to have one vertex each.
    \State Pick a boundary component $\kappa$ of $T$; it has three edges.
    \ForAll{edges $e$ of $\kappa$}
    \State Let $T_e$ be $T$ folded along $e$.
    \If{not $D^2\times S^1$?$(T_e)$}
    \State \textbf{return false}
    \EndIf
    \EndFor
    \State \textbf{return true}
    \EndProcedure
  \end{algorithmic}
\end{algorithm}
\qed
\end{Thm}

\begin{proof}
If $T$ Dehn fills to a manifold which is not
a solid torus, then $T$ is not $T^2 \times I$.
Thus the \texttt{return False} statements are correct.

It remains to show that if $T$ Dehn fills to $D^2 \times S^1$
along the three given slopes, then in fact $T$
is $T^2 \times I$. $T$ Dehn fills to $D^2 \times S^1$,
so it is the complement of a knot $k$ in $D^2 \times S^1$.
$k$ admits a nontrivial filling (actually, two fillings)
to $D^2 \times S^1$, so by Theorem 1.1 of \cite{Gabai89},
$k$ is either a 0- or 1-bridge braid. (In particular,
$T$ is irreducible, so we need not even test irreducibility
of $T$.) Now, 1-bridge braids which are not 0-bridge admit
only two $D^2 \times S^1$ fillings, by Lemma 3.2 of \cite{Gabai90}.
Therefore, $k$ is 0-bridge. Furthermore, by Example 3.1 of \cite{Gabai90},
for every nontrivial 0-bridge knot complement $N$ in $D^2 \times S^1$
with knot-neighborhood boundary $T$, there is a slope $\beta$
on $T$ such that for every slope $\alpha$ on $T$ along which
$N$ fills to a $D^2 \times S^1$, we have $\langle \alpha, \beta \rangle = 1 \mod 2$.
But we've found three $D^2 \times S^1$ slopes $a,b,c$ 
such that $a + b + c = 0 \mod 2$.
If $T$ were a nontrivial 0-bridge knot complement, then

\[
0 = \langle a + b + c, \beta \rangle
= \langle a, \beta \rangle + \langle b, \beta \rangle 
+ \langle c, \beta \rangle
= 1 + 1 + 1 = 1 \mod 2,
\]

impossible. Therefore, $T$ is a trivial 0-bridge knot complement.
That it, $T = T^2 \times I$.
\end{proof}

It remains to describe how to ``simplify'' a triangulation to induce a
minimal triangulation on a boundary components, usually called
a \emph{cusp} in this context; and how to fill along a slope in a simplified cusp.
One may find an algorithm in SnapPy for simplifying boundary components, a
special, simpler case of which is presented here.\footnote{SnapPy's approach to
cusp filling is reminiscent of the \emph{layered-triangulations} developed in \cite{JR06}.}
There is a more efficient method for accomplishing such a simplification,
but it requires the technique of \emph{crushing} developed in \cite{JR03}.

We use the following terminology.
\begin{Def}
First, suppose $M$ is materially triangulated. Let $T$, $T'$
be boundary triangles adjacent along an edge $e$.
Orient $e$ so that $T$ lies to its left
and $T'$ to its right. 
Let $\Delta$ be a fresh tetrahedron, 
and let $\tau$, $\tau'$
be boundary triangles of $\Delta$ adjacent
along an edge $\eta$. Orient $\eta$ so that
$\tau$ lies to its left and $\tau'$ to its right.
Without changing $M$'s topology we may glue
$\Delta$ to $T$ by gluing $\eta$ to $e$, 
$\tau$ to $T'$ and $\tau'$ to $T$. This is
called a \emph{two-two} move.
\qed
\end{Def}

In the above definition, the edge $\eta'$ opposite
$\eta$ in $\Delta$ becomes a boundary edge of the
new material triangulation.

\begin{Def}
We say $e$ is \emph{embedded} when its
vertices are distinct. We say $e$ is
\emph{coembedded}
when $\eta'$ as defined above is embedded. Equivalently,
$e$ is coembedded when the
vertices in $T,T'$ opposite $e$ are distinct.

Given a boundary edge $e$ between
two boundary triangles $T$ and $T'$, one may
glue $T$ to $T'$ and $e$ to itself via a valid,
orientation-reversing map from $T$ to $T'$.
In \cite{Regina} and \cite{SnapPy} this is called
the \emph{close-the-book move along $e$}.
\qed
\end{Def}

\begin{Prop}
Given a material triangulation $D$
with boundary consisting of tori,
Algorithm \ref{alg:bdy} constructs
a material triangulation $T$ of the same
underlying space such that $T$ induces a one-vertex
triangulation on every boundary component.
\qed
\end{Prop}

\begin{algorithm}
  \caption{Boundary simplification}\label{alg:bdy}
  \begin{algorithmic}[1]
    \Procedure{$\partial$-simplify}{$T$}
    \State Let $D$ be a copy of $T$
    \While{$D$ has an embedded boundary edge $e$}
        \State Layer on $e$.
        \While{$D$ has a coembedded boundary edge $f$}
            \State Close the book along $f$.
        \EndWhile
    \EndWhile
    \EndProcedure
  \end{algorithmic}
\end{algorithm}

\begin{proof}
  Two-two moves never change the topology.

  Suppose $e$ is a boundary edge of $T$ lying in distinct triangles $t$ and $t'$.
  Closing the book along $e$ will leave the topology of $T$ invariant
  if and only if the vertices opposite $e$ in $t$ and $t'$ are distinct in $T$,
  i.e.\,if and only if $e$ is coembedded.
  Notice that closing the book along a coembedded edge
  decreases the number of boundary triangles,
  and performing a two-two move on an embedded edge produces a coembedded edge
  and preserves the number of boundary triangles.
  Therefore, the above \texttt{while} loops terminate,
  using number of boundary triangles as a variant function.

  The obvious postcondition of the outer while
  loop is that there is no embedded boundary
  edge. Since the boundary is still triangulated,
  this is equivalent to each boundary component having
  only one vertex on it.
\end{proof}

\begin{Rmk}  
The routine in SnapPy is more complicated
because, rather than filling in a cusp any old way,
SnapPy wants to make sure the filling compresses
some given slope in the cusp.
\end{Rmk}

This concludes the present sketch of an algorithm
to determine hyperbolicity of a
compact, orientable 3-manifold with nonempty
boundary consisting of tori.
Both literate and raw implementations of this
algorithm as a Regina-Python module
\texttt{unhyp} reside at \cite{carrot}.
Also available at \cite{carrot} is a Regina
Python module \texttt{mom} for interpreting Milley's
data as manifolds in Regina.

\section{Two Useful Heuristics}

Although the above does yield an algorithm to
determine hyperbolicity, it frequently happens
that one can easily disprove hyperbolicity from
a group presentation. The following is a particularly
useful kind of presentation.

\begin{Def}
  A \emph{common axis relator} is a word of the form
  $a^p b^q$ or $a^p b^q a^{-p} b^{-q} (= [a^p,b^q])$ for some integers $p,q$.
  A \emph{common axis presentation} is a finite presentation
  with two generators, and with at least one common axis relator.
\end{Def}

\begin{Lem}
  A group admitting a common axis presentation is not
  the fundamental group of a hyperbolic link exterior.
\end{Lem}

\begin{proof}
  Suppose $G$ is a subgroup of $Isom^+(H^3)$ admitting
  a common axis presentation, i.e. admitting generators
  $a,b$ such that for some $p,q \in \mathbf{Z}$,
  $a^p b^q$ or $[a^p,b^q]$ is trivial. In this case,
  $a$ and $b$ must have the same axis in $H^3$.
  But then since $a$ and $b$ generate $G$,
  $G$ preserves an axis in $H^3$. If this axis
  is a line, then $a$ and $b$ are commuting loxodromic
  elements. If this axis is a point at infinity,
  then $a$ and $b$ are commuting parabolic elements.
  In both cases, $G$ is abelian, and therefore
  does not have finite-volume quotient. But
  the metric on a hyperbolic link exterior
  has finite volume.
\end{proof}

Moreover, it frequently happens that a small
triangulation of a hyperbolic 3-manifold admits
a \emph{strict angle structure}; such structures
themselves constitute proofs of hyperbolicity
(see \cite{FuterGueritaud}). Regina can very often
calculate these structures as well; this proves
hyperbolicity without a laborious normal surface enumeration.

\section{Further Directions}
Neil Hoffman and I have recently gotten results
on the complexity of the hyperbolicity problem,
which we are in the process of writing up. Among
other things we will show that the problem of
hyperbolicity is in the complexity class
\textbf{coNP} for nontrivial link exteriors,
assuming that $S^3$-recognition is in \textbf{coNP}.
\bibliographystyle{plain}
\bibliography{dethyp}

\end{document}